\documentclass{article}
\usepackage[russian,english]{babel}
\usepackage[T2A]{fontenc}
\usepackage[cp1251]{inputenc}
\usepackage{amssymb, amsthm, amsmath}
\usepackage{graphicx}
\usepackage{float}
\usepackage[usenames]{color}
\usepackage{colortbl}
\usepackage{subfigure}

\newtheorem{theorem}{Theorem}
\newtheorem{lemma}{Lemma}
\newtheorem{statement}{Statement}
\begin{document}
\sloppy
\title{Topology of 4-manifolds that admit non-singular flows with saddle orbits of the same index} 
	\author{V. Galkin, O. Pochinka}
	\maketitle
	\date{}
	
\begin{abstract}
	This paper studies regular topological flows $f^t$ defined on closed {topological} manifolds $M^n$. The chain recurrent set of such a flow consists of a finite number of topologically hyperbolic fixed points and periodic orbits. Like their smooth analogs -- Morse-Smale flows -- regular flows possess a continuous Morse-Bott function that decreases outside the chain recurrent set and is constant on the chain components of the flow. This circumstance leads to a close connection between such flows and the topology of the carrying manifold. In particular, the ambient manifold for non-singular flows (regular flows without fixed points), by the Poincare-Hopf formula, has a zero Euler characteristic. The latter property is a criterion for a manifold $M^n$ to admit a non-singular flow in all dimensions except dimension $n=3$. Thus, in higher dimensions, any odd-dimensional manifold admits a non-singular flow, and the list of even-dimensional manifolds is quite broad; at the very least, it includes all manifolds of the form $M^{n-1}\times\mathbb S^1$, where $M^{n-1}$ is any closed $(n-1)$-manifold.
	
	A surprising result of the present paper is the proof of the fact that in dimension 4, all this variety of carrying spaces can only be achieved if the flow has saddle orbits of different Morse indices. Specifically, for dimensional reasons, the Morse index of a saddle orbit of a non-singular flow $f^t: M^4 \to M^4$ can only take two values, 1 or 2. We prove that non-singular 4-flows with saddle orbits of the same Morse index exist only on skew or direct products of the 3-sphere and the circle, i.e., $M^4\cong\mathbb S^3\tilde\times\,\mathbb S^1$ or $M^4\cong\mathbb S^3\times\mathbb S^1$.
\end{abstract}  

\section{Introduction and formulation of results}
Let $M^n,\,n\geqslant 2$ be a connected closed { topological} $n$-manifold with a metric $d$. Recall that {\it a flow} on the manifold $M^n$ is defined as a family of homeomorphisms $f^t:M^n\to M^n$ continuously dependent on $t\in\mathbb R$ with the group properties:
\begin{itemize}
	\item[1)] $f^0(x)=x$ for any $x\in M^n$;
	\item[2)] $f^t(f^s(x))=f^{t+s}(x)$ for any $s,t\in \mathbb{R}$, $x\in M^n$. 
\end{itemize}
{\it An orbit} of a point $x\in M^n$ for the flow $f^t$ is the set $\mathcal O_x=\{f^t(x), t\in \mathbb{R}\}$, and the orbit $\mathcal O_x$ is called:
\begin{itemize}
\item {\it fixed point}, if $\mathcal O_x=\{x\}$,
\item {\it periodic orbit}, if there are $0<t_1<t_2$ such that $f^{t_1}(x)\neq x$ и $f^{t_2}(x)=x$,
\item {\it regular orbit}, if $f^t(x)\neq x$ for any $t\neq 0$.
\end{itemize} 
All trajectories of the flow, other than a fixed point, are considered to be oriented in accordance with the increasing parameter $t$.

Flows $f^t: M^n\to M^n$ and $f'^{t}:M^n\to M^n$ are called {\it topologically equivalent} if there exists a homeomorphism $h: M^n\to M^n$ that maps the trajectories of the flow $f^t$ to the trajectories of the flow $f'^{t}$  preserving orientation.

For fixed points and periodic orbits of a topological flow, the concept of hyperbolicity is defined as follows. For $q\in\mathbb N$, let 
$$\mathbb B^q=\{(x_1,\dots,x_q) \in \mathbb R^q:\,x_1^2+\dots,+x_q^2\leqslant 1\}\,-\,q\hbox{-ball};$$
$$\mathbb S^{q-1}=\partial\mathbb B^q\,-\,(q-1)\hbox{-sphere};$$
$$\mathbb B^{q}\,\tilde{\times}\,\mathbb S^1\,-\,\hbox{skew product of the q-ball and the circle},$$
that is, the manifold obtained from $\mathbb B^{q}\times[0,1]$ by identifying the balls $\mathbb B^{q}\times{0},\,\mathbb B^{q}\times{1}$ via an orientation-reversing homeomorphism; 
$$\mathbb S^{q-1}\tilde\times\,\mathbb S^1=\partial(\mathbb B^{q}\,\tilde{\times}\,\mathbb S^1).$$ 
Furthermore, the symbol $\otimes$ will denote either $\times$ or $\tilde\times$.

{ Consider a linear flow $a^t_{i}:\mathbb R^n\to\mathbb R^n,\,i\in\{0,\dots,n\},$ given by the formula 
\begin{equation*}
 a ^ t_ {i} (x_1,\dots, x_ {i}, x_ {i+ 1},\dots, x_n) = (2 ^ tx_1,\dots, 2 ^ tx_ {i}, 2 ^ {- t} x_ {i+ 1},\dots,2^{- t} x_n). 
\end{equation*}
By construction, $O(0,\dots,0)$ is a fixed point of the flow $a^t_i$. Let 
\begin{equation*}
 V_{i}=\{(x_1,\dots, x_ {i}, x_ {i+ 1},\dots, x_n)\in\mathbb R^n:\,x^2_1+\dots+x_ {i}^2\leqslant 1,\,x^2_{i+ 1}+\dots+ x_n^2\leqslant 1\},
\end{equation*}
\begin{equation*}
\Sigma^u_{i}= \{(x_1,\dots, x_ {i}, x_ {i+ 1},\dots, x_n)\in\mathbb R^n:\,x^2_1+\dots+x_ {i}^2= 1,\,x^2_{i+ 1}+\dots+ x_n^2\leqslant 1\},
\end{equation*}
\begin{equation*}
	\bar\Sigma^s_{i}= \{(x_1,\dots, x_ {i}, x_ {i+ 1},\dots, x_n)\in\mathbb R^n:\,x^2_1+\dots+x_ {i}^2\leqslant 1,\,x^2_{i+ 1}+\dots+ x_n^2= 1\}.
\end{equation*}
Then the neighborhood $V_i$ of the fixed point $O$ is an {\it $n$-dimensional $i$-handle} 
\begin{equation*}
V_i\cong\mathbb B^i\times\mathbb B^{n-i}=\mathbb B^n
\end{equation*}
with the boundary $$\Sigma_{i}=\Sigma^u_{i}\cup\Sigma^s_{i}\cong \mathbb S^{n-1},$$ part $\Sigma^u_{i}$ of which is called {\it the base of the handle $V_{i}$}.

A fixed point $p$ of a topological flow $f^t$ is called {\it hyperbolic} if there exists its neighborhood $V_{p}\subset M^n$, a number $i_p \in \{0,\dots, n \}$, and a homeomorphism $h_{p}:V_{i_p}\to V_{p}$, establishing the equivalence of the flow $a ^ t_ {i_p}|_{V_{i_p}}$ with the flow $f^t|_{V_{p}}$.  

The neighborhood $V_{p}$ is called {\it canonical} in this case. Its boundary $\Sigma_{p}$ consists of {\it the base} $\Sigma^u_p=H_{p}(\Sigma^u_{i_{p}})$ and the set $\Sigma^s_p=H_{p}(\Sigma^s_{i_p})$. The number $i_p$ is called {\it Morse index} of $p$. Fixed points of indices $n$ and $0$ are called {\it source} and {\it sink} points, respectively, and those with indices $0<i_p<n$ are called {\it saddle} points (see Fig. \ref{nodes+saddles-new}).} 
\begin{figure}[ht]
\centerline{\includegraphics
[height=4cm]{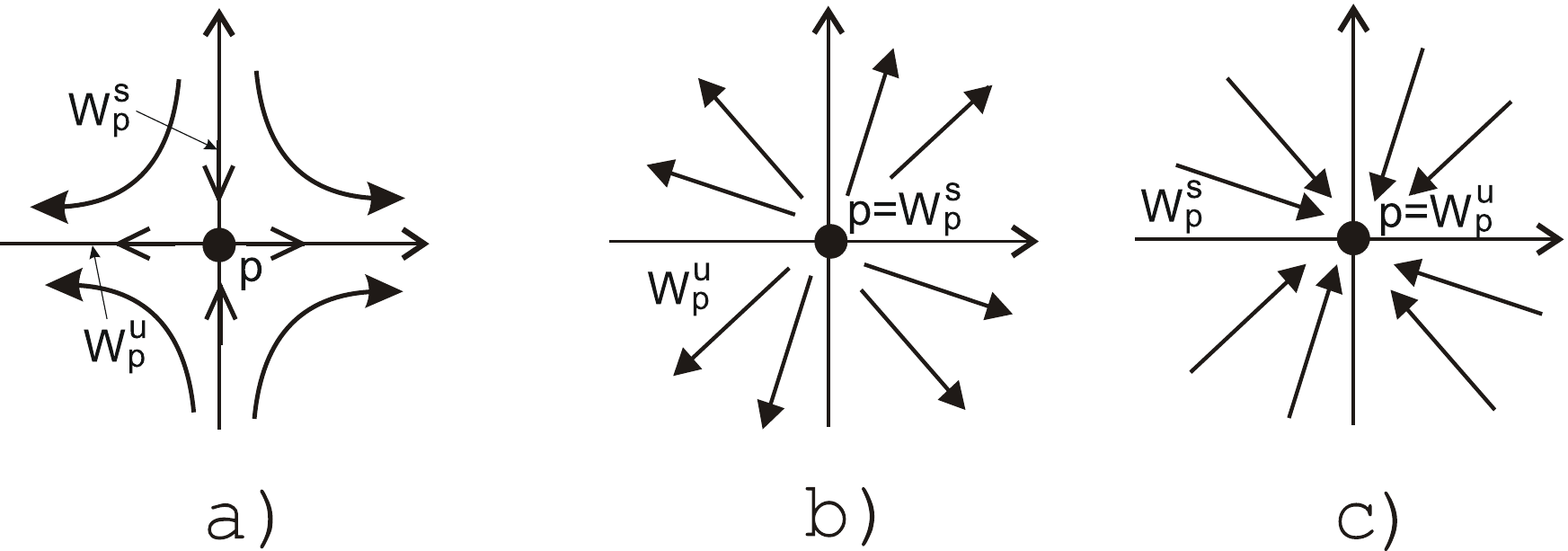}}
\caption{The dynamics in the neighborhood of a hyperbolic fixed point: a) saddle point, b) source point, c) sink point} \label{nodes+saddles-new}
\end{figure}

To define the hyperbolicity of a periodic orbit, let us recall the notion of a suspension. Let $a\colon M\to M$ be a homeomorphism of a topological manifold $M$. Define the homeomorphism $g_{a}\colon M\times \mathbb R \to M\times \mathbb R$ by the formula $$g_{a}(x,r) = (a(x),r-1).$$ 
Then the group of homomorphisms ${g_{a}^k}\cong\mathbb Z$ acts freely and properly discontinuously\footnote{A group of homomorphisms $G$ of a topological space $X$ {\it acts freely} on $X$ if $g(x)\neq x$ for all $x \in X$ and all $g \in G,,g\neq id$. The group $G$ {\it acts properly discontinuously} on $X$ if for every compact set $K \subset X$, the set of those $g\in G$ for which $g(K)\cap K \neq \varnothing$ is finite.} on $M\times \mathbb R$, whereby (see, for example, \cite{Ko}) the orbit space $M_a = (M\times \mathbb R)/ g_{a}$ is a topological manifold, and the natural projection $$v_a\colon M\times \mathbb R\to M_a$$ {is a covering map}. Moreover, the flow $\xi^t\colon M\times \mathbb R\to M\times \mathbb R$, defined by the formula  $$\xi^t(x,r)=(x,r+t),$$ induces a unique flow  $$[a]^t\colon M_\varphi\to M_\varphi,$$  satisfying the condition $[a]^tv_a= v_a \xi^t$ and called the {\it suspension over the diffeomorphism $a$}.

For $i\in\{0,\dots,n-1\},\,\rho,\,\nu\in\{-1,+1\}$ define the diffeomorphism 
$a_{i,\rho,\nu}\colon \mathbb R^{n-1}\to \mathbb R^{n-1}$ by formula 
\begin{equation*}\label{air}
\begin{gathered} 
a_{i,\rho,\nu}(x_1,x_2,\dots, x_i,x_{i+1},x_{i+2},\dots, x_{n-1}) =\\=\left (2\rho x_1,2x_2,\dots 2x_i,\frac{\nu x_{i+1}}{2},\frac{x_{i+2}}{2},\dots,\frac{x_{n-1}}{2}\right).
\end{gathered}
\end{equation*}
Let (see Fig. \ref{barV})
\begin{equation*}\label{vir}
	\bar V_{i}= \{ (x_1,\dots,x_{n-1},r)\in \mathbb R^{n-1}\times\mathbb R :\, 
	x_1^2+\dots+ x^2_i\leqslant 4^{-r}, x^2_{i+1}+\dots+x_{n-1}^2 \leqslant 4^r\},
\end{equation*}
\begin{equation*}
	\bar\Sigma^u_{i}= \{ (x_1,\dots,x_{n-1},r)\in \mathbb R^{n-1}\times\mathbb R :\, 
	x_1^2+\dots+ x^2_i= 4^{-r}, x^2_{i+1}+\dots+x_{n-1}^2 \leqslant 4^r\},
\end{equation*}
\begin{equation*}
	\bar\Sigma^s_{i}= \{ (x_1,\dots,x_{n-1},r)\in \mathbb R^{n-1}\times\mathbb R :\, 
	x_1^2+\dots+ x^2_i\leqslant 4^{-r}, x^2_{i+1}+\dots+x_{n-1}^2= 4^r\},
\end{equation*}
\begin{equation*}\label{virn}
O_{i,\rho,\nu}=v_{a_{i,\rho,\nu}}(Or),V_{i,\rho,\nu}=v_{a_{i,\rho,\nu}}(\bar V_{i}),\,\Sigma^u_{i,\rho,\nu}=v_{a_{i,\rho,\nu}}(\bar\Sigma^u_{i}),\,\Sigma^s_{i,\rho,\nu}=v_{a_{i,\rho,\nu}}(\bar\Sigma^s_{i}).
\end{equation*}
\begin{figure}[h!]
\centerline{\includegraphics
[height=5cm]{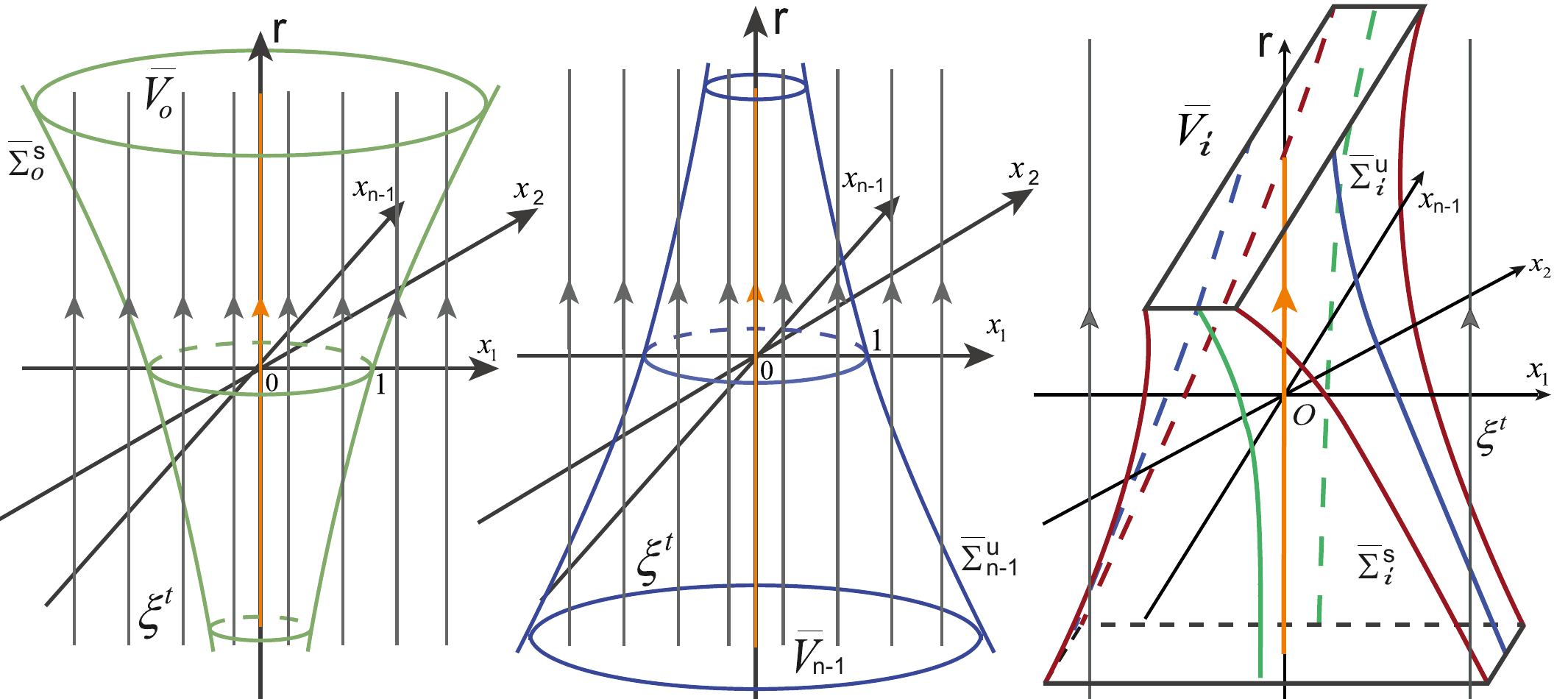}}
\caption{Sets $\bar V_{0},\bar V_{n-1},\bar V_i,0<i<n-1$} \label{barV}
\end{figure}
Then $O_{i,\rho,\nu}$ is a periodic orbit of the flow $[a_{i,\rho,\nu}]^t$. Its neighborhood $V_{i,\rho,\nu}$ is an {\it $n$-dimensional round $i$-handle}  
\begin{equation*}\label{round}
V_{i,\rho,\nu}\cong\mathbb B^{n-1}\otimes\mathbb S^1
\end{equation*}
with boundary $$\Sigma_{i,\rho,\nu}=\Sigma^u_{i,\rho,\nu}\cup\Sigma^s_{i,\rho,\nu}\cong \mathbb S^{n-2}\otimes\mathbb S^1,$$ part $\Sigma^u_{i,\rho,\nu}$ of which is called {\it the base of the round handle $V_{i,\rho,\nu}$}. If $\rho=1,(\rho=-1)$, then the round handle $V_{i,\rho,\nu}$ is called {\it untwisted (twisted)}.

A periodic orbit $\mathcal O$ of a topological flow $f^t:M^n\to M^n$ is called {\it hyperbolic} if there exists its neighborhood $V_{\mathcal O}\subset M^n$, numbers $i_{\mathcal O}\in \{0,1,\dots, n-1 \},\,\rho_{\mathcal O},\nu_{\mathcal O}\in\{-1,+1\}$ and a homeomorphism $H_{\mathcal O}: V_{i_{\mathcal O},\rho_{\mathcal O},\nu_{\mathcal O}}\to V_{\mathcal O}$, establishing the equivalence of the flow $[a_{i_{\mathcal O},\rho_{\mathcal O},\nu_{\mathcal O}}]^t$ with the flow $f^t|_{V_{\mathcal O}}$.

The neighborhood $V_{\mathcal O}$ is called {\it canonical}, its boundary $\Sigma_{\mathcal O}$ consists of the {\it base} $\Sigma^u_\mathcal O=H_{\mathcal O}(\Sigma^u_{i_{\mathcal O},\rho_{\mathcal O},\nu_{\mathcal O}})$ and the set $\Sigma^s_\mathcal O=H_{\mathcal O}(\Sigma^s_{i_{\mathcal O},\rho_{\mathcal O},\nu_{\mathcal O}})$. The number $i_{\mathcal O}$ is called the {\it Morse index} of the hyperbolic periodic orbit $\mathcal O$. Periodic orbits with Morse indices $n-1$ and $0$ are called {\it repelling} and {\it attracting}, respectively, and otherwise -- {\it saddle}. A periodic orbit is called {\it untwisted (twisted)} if its canonical neighborhood is homeomorphic to an untwisted (twisted) round handle.

A topological flow $f^t$ is called {\it regular} if its chain recurrent set $\Omega_{f^t}$ consists of a finite number of connected components -- {\it chain components}, each of which is a hyperbolic fixed point or a hyperbolic periodic orbit. The {\it stable, unstable $($invariant$)$}, respectively, manifolds of a chain component $C$ are the sets
\begin{gather*}
W^s_{C}=\left\{x\in M^n:\lim\limits_{t\to+\infty}\left(\min_{y\in C}d(y,f^{t}(x))\right)=0\right\},\\
W^u_{C}=\left\{x\in M^n:\lim\limits_{t\to-\infty}\left(\min_{y\in C}d(y,f^{t}(x))\right)=0\right\}.
\end{gather*} 
An {\it heteroclinic intersection} is the intersection of invariant manifolds of different saddle chain components. 

A regular flow is called {\it non-singular} if it has no fixed points.

{ Regular flow is a topological generalization of the Morse-Smale flow. Their appearance (see, for example, \cite{8}) is motivated by two factors: 1) the existence of topological manifolds of dimension 4 and higher that lack a smooth structure; 2) the development of methods for the topological classification of smooth systems, which utilize purely topological properties of these systems and thus allow the inclusion of systems that lack smoothness.

The properties of regular flows with a four-dimensional ambient manifold essentially rely on results of three-dimensional topology, which, in particular, makes it possible to establish the following result.

\begin{theorem}\label{smoothing} If a closed topological 4-manifold admits a regular flow, then it is smoothable.
\end{theorem}}

The main part of the paper is devoted to studying the topology of the carrying manifolds of non-singular four-dimensional flows. To formulate the results obtained by the authors, we outline the already known facts in this research area.

Let $f^t:M^n\to M^n$ be a non-singular flow. By the Poincare-Hopf theorem, the Euler characteristic $\chi(M^n)$ of a manifold $M^n$ that admits a non-singular flow is zero. Smooth analogues of non-singular flows are called {\it NMS-flows} (Morse-Smale flows without fixed points).

D. Asimov \cite{Asimov75} showed that a smooth closed $n$-manifold $M^n$, $n > 3$, admits an NMS-flow $\iff$ it admits a  decomposition on the untwisted round handles $\iff \chi(M^n)=0$.

On two-dimensional surfaces only the torus and the Klein bottle admit non-singular flows, and their classification follows, for example, from the works of M. Peixoto, A. Oshemkov, V. Sharko, D. Malyshev, V. Kruglov, O. Pochinka \cite{Peix}, \cite{OS}, \cite{MaKrPo_OmSt}.

J. Morgan \cite{Morgan79} proved that an irreducible closed orientable 3-manifold admits an NMS-flow if and only if it is a graph manifold. From this, it follows, for example, that hyperbolic 3-manifolds do not admit NMS-flows. The classification of three-dimensional NMS-flows has been obtained only for some special cases. Necessary and sufficient conditions for the equivalence of NMS-flows, in which at least one invariant manifold of any saddle orbit does not contain heteroclinic intersections, follow from the Y. Umansky paper  \cite{Umansky}. However, these conditions do not allow one to judge the admissibility of a particular flow on a given manifold. In the case of a small number of periodic orbits, an exhaustive classification of non-singular flows follows from the works of B. Yu, O. Pochinka, D. Shubin \cite{Yu}, \cite{PoSh22}, \cite{PoSh23}, \cite{PoSh24}. In particular, the results of these works imply that the carrying 3-manifold of an NMS-flow with no more than one saddle orbit is homeomorphic to either a lens space, a connected sum of lens spaces, or a small Seifert manifold.

In the work of O. Pochinka and D. Shubin \cite{PoSh22}, a complete topological classification of non-singular $n$-flows ($n>3$) without saddle orbits was obtained. It was established that the carrying manifold of such flows is homeomorphic to $\mathbb S^{n-1}\otimes\,\mathbb S^1$, and each such manifold admits exactly two equivalence classes of the considered flows.

In the works of V. Galkin, O. Pochinka, and D. Shubin \cite{GPS24}, \cite{GP25}, \cite{GPS25}, a complete topological classification of non-singular flows without heteroclinic intersections on closed orientable 4-manifolds was obtained.

For dimensional reasons, the Morse index of a saddle orbit of a non-singular flow $f^t: M^4 \to M^4$ can only take two values, 1 or 2. At the same time, the list of 4-manifolds that admit non-singular flows is quite extensive; at the very least, it includes all manifolds of the form $M^{3}\times\mathbb S^1$, where $M^{3}$ is any closed $3$-manifold. A surprising result of the present paper is the proof of the following fact.

\begin{theorem}\label{s3s1s} $\mathbb S^3\times\mathbb S^1\,(\mathbb S^3\tilde\times\,\mathbb S^1)$ is the only connected closed orientable (non-orientable) 4-manifold that admits a non-singular flow, all of whose saddle orbits have the same Morse index. 
\end{theorem}

Note that Theorem \ref{s3s1s} for flows that are suspensions over regular 3-homeomorphisms without periodic orbits and with saddle points of the same index directly follows from the work \cite{OP24}, which established that all such homeomorphisms are admitted only by the 3-sphere. However, not all flows considered in Theorem \ref{s3s1s} are suspensions. In particular, according to \cite{OP24}, for any suspension, the number $k_1$ of saddle orbits and the number $k_0$ of attracting orbits are related by the equality $k_1=k_0-1$. Meanwhile, as follows from the proof of Theorem \ref{s3s1s}, in the general case, the number of saddle orbits can be any number satisfying the inequality $k_1\geqslant k_0-1$.

{\bf Acknowledgments.} {The work was supported by project ``Symmetry. Information. Chaos'' within the framework of Program of fundamental research of HSE in 2025. The authors thank P.M. Akhmetiev and E.Ya. Gurevich for useful discussions.}

\section{Necessary information about regular flows} 
Recall that a topological flow $f^t$, given on a closed connected topological manifold $M^n$, is called {\it regular} if its chain-recurrent set $\Omega_{f^t}$ consists of a finite number chain components, each of which is a hyperbolic fixed point or a hyperbolic periodic orbit. By introducing the Smale relation on the set $\{C_j\}$ of chain components of the flow $f^t$ by the condition 
\begin{equation}\label{Smor}
C_j\prec C_{j'}\iff W^s_{C_j}\cap W^u_{C_{j'}}\neq\varnothing,
\end{equation} 
one can prove (see, for example, \cite[Lemma 2]{PoZi_RCD}) the absence of {\it cycles} -- sets of pairwise distinct components $C_{1},C_{2},\cdots, C_{m}$ satisfying the condition $$C_{1}\prec C_{2}\prec\cdots\prec C_{m}\prec C_{1}.$$
Then, according to the Szpilrajn theorem \cite{Szp}, the relation \eqref{Smor} can be extended by transitivity to a total order relation.

\begin{statement}[\cite{PoZi_RCD}, Theorem 1)]\label{main} Let $f^t:M^n\to M^n$ be a regular flow with an ordered set of chain components $C_1\prec\dots\prec C_k.$ Then
\begin{enumerate}
\item[1)] $
M^n=\bigcup\limits_{j=1}^kW^u_{C_{j}}=\bigcup\limits_{j=1}^kW^s_{C_{j}}$; 
\item[2)]  unstable $W^u_{p}$ $($stable $W^s_{p})$ manifold of the fixed point $C_j=p$ is a topological submanifold of the manifold $M^n$, such that
$W^u_{p}\cong\mathbb R^{i_{p}}\,(W^s_p\cong\mathbb R^{n-i_{p}})$;
\item[3)] unstablw $W^u_{\mathcal O}$ $($stable $W^s_{\mathcal O})$ the manifold of the periodic orbit $C_j=\mathcal O$ is a topological submanifold of the manifold $M^n$ such that

$W^u_{\mathcal O}\cong\mathbb R^{i_{\mathcal O}}\times\mathbb S^1\,(W^s_{\mathcal O}\cong\mathbb R^{n-i_{\mathcal O}-1}\times\mathbb S^1),\,\rho_{\mathcal O}=+1\,(\nu_{\mathcal O}=+1),$

$W^u_{\mathcal O}\cong\mathbb R^{i_{\mathcal O}}\widetilde{\times}\,\mathbb S^1\,(W^s_{\mathcal O}\cong\mathbb R^{n-i_{\mathcal O}-1}\widetilde{\times}\,\mathbb S^1),\,\rho_{\mathcal O}=-1\,(\nu_{\mathcal O}=-1)$; 
\item[4)] 
 ${\rm cl}(W^u_{C_j})\setminus W^u_{C_j}\subset\bigcup\limits_{l=1}^{j-1}W^u_{C_l}\,\,({\rm cl}(W^s_{C_j})\setminus W^s_{C_j}\subset\bigcup\limits_{l=j+1}^{k}W^s_{C_l})$. 
\end{enumerate}
\end{statement}
	
\begin{statement}[\cite{PoZi_RCD}, Theorem 2]  Any regular flow $f^t:M^n\to M^n$ have a continuous Morse-Bott energy function $\varphi:M^n\to\mathbb R$, whose critical points are either non-degenerate and have the same index as the corresponding fixed point, or form non-degenerate one-dimensional manifolds and have the same index as the corresponding periodic orbit. Moreover, for any fixed order $C_1\prec\dots\prec C_k$ of the chain components of the flow $f^t$, the function $\varphi$ can be constructed such that $\varphi(C_j)=j,\,j\in \{1,\dots, k\}$.   
\label{enflow}
\end{statement}

Let $\varphi: M^n\to \mathbb R$ be a Morse-Bott energy function of a regular flow $f^t:M^n\to M^n$ such that $\varphi(C_j)=j$ for the ordered chain components $C_1\prec\dots\prec C_k$. For $j\in {1,\dots, k-1}$, let $$M_{j}=\varphi^{-1}\left(1,j+\frac12\right).$$ From the properties of the Morse-Bott function, it follows that the filtration  
\begin{equation}\label{filtr}
\varnothing=M_0\subset M_1\subset M_2\subset\dots \subset M_{k-1}\subset M_k=M^n
\end{equation}
is a  decomposition of the manifold $M^n$ into round handles, such that
\begin{equation}\label{rouhan}
M_j\cong M_{j-1}\cup_{\psi_j} V_{{C_j}}
\end{equation} 
for some embedding $\psi_j:\Sigma^u_{C_j}\to\partial M_{j-1}$.  

\begin{statement}[\cite{Franks78}, Theorem A]\label{thmA}
Let $k_i$ be the number of untwisted round $i$-handles in the decomposition \eqref{filtr} and 
$\beta_i = \dim H_i(M^n, \mathbb Z).$  Then

(a) $k_0\geq \beta_0$, $k_i\geq \beta_i-\beta_{i-1}+\dots\pm\beta_0$, $i\geq 1$;

(b) $k_1\geq k_0-1$, $k_{n-2}\geq k_{n-1}-1$;

(c) if $k_{i-1}=k_{i+1}=0$ and $\beta_i-\beta_{i-1}+\dots\pm\beta_0\leq 0$, then $k_i=0$.
\end{statement}

\section{Smoothability of the carrying manifold of a regular 4-flow}

This section is devoted to the proof of Theorem \ref{smoothing}: If a closed topological 4-manifold admits a regular flow, then it is smoothable.

\begin{proof} Let $C_1,\dots, C_k$ be the ordered chain components of the the regular flow $f^t:M^4\to M^4$. From the definition of a continuous Morse-Bott function, it follows that  the boundary of each set $M_j$ in the filtration \eqref{filtr} is a locally flat three-dimensional submanifold of the manifold $M^4$. Moreover, by \eqref{rouhan}, the closure of the set $M_{j+1}\setminus M_j$ is homeomorphic to the canonical neighborhood $V_{C_j}$, and hence homeomorphic to one of the manifolds $\mathbb B^4$, $\mathbb B^{3}\otimes\mathbb S^1$. Since the listed manifolds are smoothable, the manifold $M^4$ is obtained by gluing smooth compact manifolds $$M^4=M_1\cup\, {\rm cl}\,(M_2\setminus M_1)\cup \cdots \cup \,{\rm cl}\,(M_k\setminus M_{k-1})$$along some homeomorphisms on their three-dimensional boundaries. According to \cite{Munc60}, any homeomorphism of a three-dimensional manifold can be approximated by a diffeomorphism, which, by \cite{Ch69}, is isotopic to it. Since gluing two manifolds along isotopic homeomorphisms yields homeomorphic manifolds, {by replacing the attaching homeomorphisms with diffeomorphisms isotopic to them}, we obtain a smooth copy of the manifold $M^4$.
\end{proof}
 
\section{Reducing the proof of Theorem \ref{s3s1s} to the orientable case}\label{ori}
Let $f^t:M^4\to M^4$ be a non-singular flow, all of whose saddle orbits have the same Morse index. Since this index can be 1 or 2, without loss of generality, we will  consider flows all of whose saddle orbits have Morse index 1 (the case where all saddle orbits have Morse index 2 reduces to the considered one by time reversal).

According to Theorem \ref{smoothing}, the manifold $M^4$ is smoothable. Then, in the case of non-orientability of $M^4$, there exists (see, for example, \cite[Lemmas 5.1, 5.2]{BPY}) a two-fold covering $p:\bar M^4\to M^4$, where $\bar M^4$ is a connected closed orientable 4-manifold that admits a non-singular flow $\bar f^t:\bar M^4\to\bar M^4$, all of whose saddle orbits have Morse index 1, covering the flow $f^t$, that is,
$p\bar f^t=f^tp.$  
\begin{lemma}\label{neskew} Let $M^4$ be a connected closed non-orientable topological 4-manifold that admits a non-singular flow $f^t: M^4\to M^4$, all of whose saddle orbits have Morse index 1. Then, if $\bar M^4\cong\mathbb S^3\times\mathbb S^1$, then $M^4\cong\mathbb S^3\tilde\times\,\mathbb S^1$.
\end{lemma}
\begin{proof} Since $p:\bar M^4\to M^4$ is a two-fold covering, $M^4$ is the orbit space $\bar M^4/\iota$ of the action of a fixed points free involution $\iota:\bar M^4\to \bar M^4$. Since $\bar M^4\cong\mathbb S^3\times\mathbb S^1$, according to \cite{Kwasik}, $\bar M^4/\iota$ is homeomorphic to one of the four manifolds $\mathbb S^3\times\mathbb S^1,\,\mathbb S^3\tilde\times\mathbb S^1,\,\mathbb RP^3\times\mathbb S^1,\,\mathbb RP^4\# \mathbb RP^4$, where $\mathbb RP^n$ is $n$-dimensional real projective space. Among the listed spaces, the non-orientable ones are  
	$\mathbb S^3\tilde\times\mathbb S^1$ and $\mathbb RP^4\# \mathbb RP^4$. Let us show that the manifold $\mathbb RP^4\# \mathbb RP^4$ does not admit non-singular flows all of whose saddle orbits have Morse index 1.

Assume the contrary: there exists such a  flow $f^t:\mathbb RP^4\# \mathbb RP^4\to \mathbb RP^4\# \mathbb RP^4$. Since the manifold $\mathbb RP^4\# \mathbb RP^4$ has Betti numbers $\beta_0=\beta_2=\beta_4=1$, $\beta_1=\beta_3=0$, it follows from item (a) of Statement \ref{thmA} that the number $k_2$ of untwisted saddle orbits of index 2 satisfies the estimate: $k_2\geqslant\beta_2-\beta_1+\beta_0=2$. This contradicts the assumption.
\end{proof}

\section{Dynamically ordered filtration of the flow $f^t$}
In this section, we establish the structure of the periodic orbits of a non-singular flow $f^t:M^4\to M^4$, all of whose saddle orbits have Morse index 1. Taking into account Lemma \ref{neskew}, we will assume that $M^4$ is an orientable manifold. In this case, all attracting and repelling orbits of the flow are untwisted. Without loss of generality, we can assume that all saddle orbits of the flow $f^t$ are untwisted also, since otherwise this can be achieved by applying a period-doubling bifurcation to each saddle orbit with a non-orientable unstable manifold (see Fig. \ref{flip}). 
\begin{figure}[h!]
\centerline{\includegraphics
[height=3.4cm]{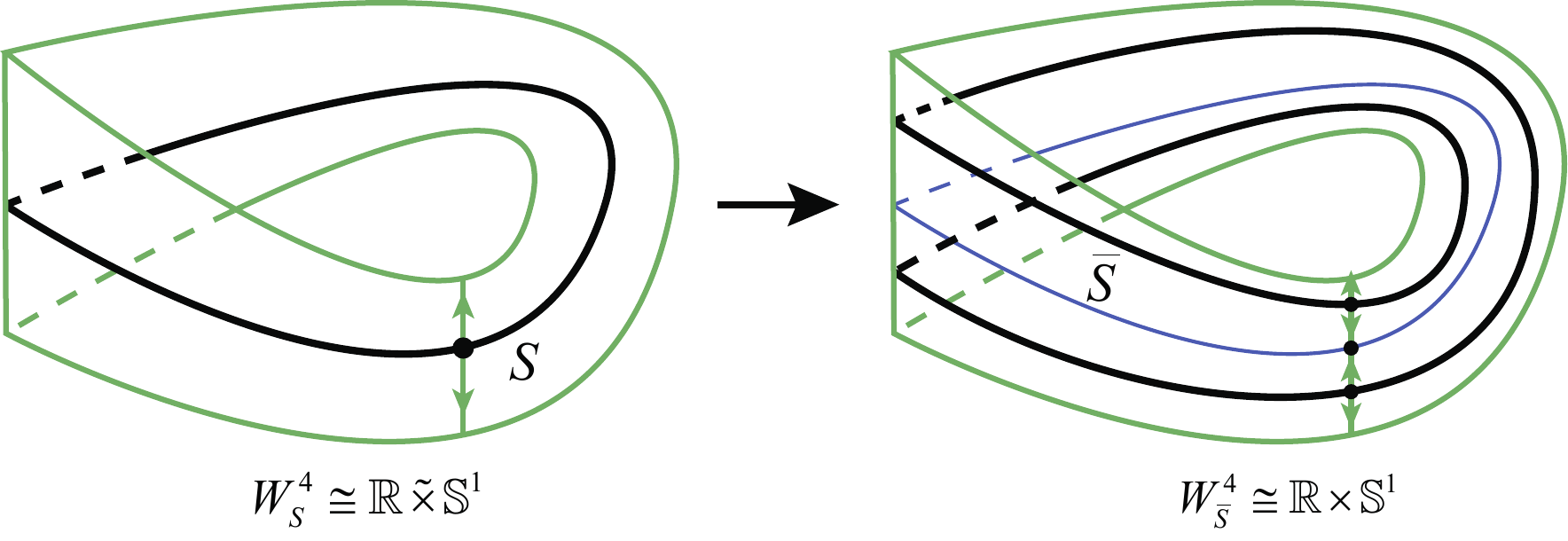}}
\caption{Period-doubling bifurcation of a saddle orbit $S$} \label{flip}
\end{figure} 
Then $k_i,\,i\in\{0,1,2,3\}$ is the number of the periodic orbits of the flow $f^t$ with the Morse index $i$.  
As $\beta_0=1$ due to the connectedness of the manifold $M^4$, then Statement \ref{thmA} gives  the following structure of the periodic orbits of the flow $f^t$:
\begin{equation}\label{N0}
k_0\geqslant 1,
\end{equation} 
\begin{equation}\label{N1}
k_1\geqslant k_0-1,
\end{equation}
\begin{equation}\label{N2}
k_2=0,
\end{equation}
\begin{equation}\label{N3}
k_3=1.
\end{equation}

From relations \eqref{N0}-\eqref{N3} it follows that the total number of periodic orbits of the flow $f^t$ is $k=k_0+k_1+1$. Let us renumber the periodic orbits in order of increasing Morse index and such that the introduced order, called {\it dynamic}, satisfies the Smale relation:
\begin{equation}\label{dynod}
\mathcal O_1\prec\dots\prec\mathcal O_{k_0}\prec \mathcal O_{k_0+1}\prec\dots\prec\mathcal O_{k-1}\prec \mathcal O_{k},
\end{equation}
where $\mathcal O_j,\,j\in\{1,\dots,k_0\}$ -- attracting orbits, $\mathcal O_j,\,j\in\{k_0+1,\dots,k-1\}$ -- saddle orbits, $\mathcal O_{k}$ -- repelling orbit. Let 
\begin{equation}\label{fil}
\varnothing=M_0\subset M_1\subset M_2\subset\dots \subset M_{k-1}\subset M_k=M^4\, 
\end{equation} be the round handle  decomposition corresponding to the order \eqref{dynod}.

To study the topology of the filtration sets in \eqref{fil}, we will introduce some notation and prove auxiliary facts.

\section{Surgery of a 3-manifold along a 2-torus}
Let $L_{p,q}$ denote the lens space, $E=\mathbb S^2\times\mathbb S^1=L_{0,1}$, $P^{n+1}=\mathbb B^{n}\times\mathbb S^1,n\in\mathbb N$. By $(M)_{p,q}$ we will denote the 3-manifold obtained from a closed orientable 3-manifold $M$ by $(p,q)$-Dehn surgery along some framed  knot in $M$.     

Let $A$ be a closed connected orientable 3-manifold and $e:\mathbb T^2\times\mathbb B^1\to A$ be an embedding. Then $T=e(\mathbb T^2\times\{0\})$ is a two-dimensional torus with a tubular neighborhood $N(T)=e(\mathbb T^2\times\mathbb B^1)$  with boundaries  $T_-=e(\mathbb T^2\times\{-1\})$ and $T_+=e(\mathbb T^2\times\{1\})$. Let  $\bar  A=(P^3\times\mathbb S^0)\sqcup (A\setminus{\rm int}\,N(T))$  и $\psi:\partial P^3\times\mathbb S^0\to \partial N(T)$ such a homeomorphism that for the meridian $c\subset \partial P^3$ the knots $\psi(c\times\{-1\})\subset T_-,\,\psi(c\times\{1\})\subset T_+$ are homotopic in $N(T)$. The manifold
\begin{equation}\label{P12}
A^T=\bar A/_{\psi},
\end{equation}
where $\bar A/_{\psi}$ -- the quotient space obtained from $\bar A$ by the minimal equivalence relation $\sim$ for which $x\sim\psi(x)$, is called {\it surgery of $A$ along the torus $T$}. Denote by $p:\bar A\to A^T$ the natural projection. The solid tori $P_T^1=p(P^3\times\{-1\})$, {  $P_T^2=p(P^3\times\{1\})$} in the manifold $A^T$ is called {\it associated with $T$}.

In the case when the torus $T$ is compressible in the manifold $A$, there is an {\it associated with $T$ 2-sphere} $\Lambda_T\subset A$, constructed as follows. According to \cite[Lemma 6.1]{He}, there exists a non-contractible knot $c\subset T$ and a 2-disk $D\subset A$ with a tubular neighborhood $N(D)$ such that $D\cap T=\partial D=c$. By construction, the boundary of the set $N(T)\cup N(D)$ consists of a 2-sphere $\Lambda_T$ and a 2-torus $T_0$ (see Fig. \ref{pict818}), and $\Lambda_T\cap T=\varnothing$.
\begin{figure}[ht]
\centering{\includegraphics
[scale=0.4]{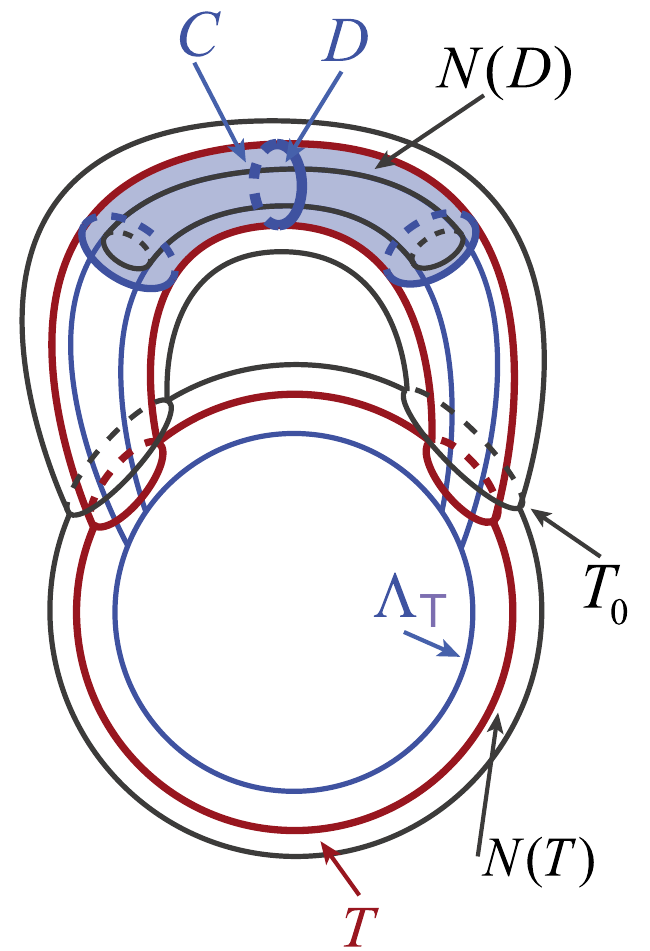}}
\caption{Sphere $\Lambda_T$}\label{pict818}
\end{figure}	
 
Denote by $N(\Lambda_T)$ the tubular neighborhood of the sphere $\Lambda_T$.  
There are two possibilities: 1) $\Lambda_T$ divides $A$, 2) $\Lambda_T$ does not divide $A$. In case 1), the manifold $A\setminus{\rm int}\,N(\Lambda_T)$ consists of two connected components $A_1, A_2$, exactly one of which (assume for definiteness $A_2$) contains the disk $D$. Denote by $\tilde A_i,i=1,2$ the manifold obtained from the manifold $A_i$ by attaching a 3-ball $B_i$ to its boundary. In case 2), denote by $\tilde A$ the manifold obtained from the manifold $A\setminus{\rm int}\,N(\Lambda_T)$ by attaching two 3-balls $B_1,B_2$ to its boundary.

\begin{lemma}\label{decoo}In cases 1), 2), the manifolds $A, A^T$ have the following forms, respectively:

1) $A\cong\tilde A_1\#\tilde A_2$, $A^T\cong L_{p,q}\#\tilde A_1\sqcup(\tilde A_2)_{p,q}$;

2) $A\cong E\#\tilde A$, $A^T\cong L_{p,q}\#(\tilde A)_{p,q}$.

Moreover, one of the associated solid tori (assume $P^2_T$ for definiteness) is a subset of $L_{p,q}$ such that $L_{p,q}\setminus{\rm int}\,P^2_T\cong P^3$.
\end{lemma}
\begin{proof} The representation of the manifold $A$ in the form 1) $A\cong\tilde A_1\#\tilde A_2$, 2) $A\cong E\#\tilde A$ in the cases where the sphere $\Lambda_T$ does or does not divide it, respectively, is a classical result of three-dimensional topology (see, for example, \cite[Theorem 1.5]{Hat}). In case 1), $B_2\cup N(D)$ is a solid torus in $\tilde A_2$ and, consequently, the tori $T_-,T_+$ bound solid tori $P_-,P_+$ in $\tilde A_2$. For definiteness, assume that $B_2\subset P_-\subset P_+$. Then $P_-\cup_\psi(P^3\times\{-1\})$ is a lens space $L_{p,q}$ and $(\tilde A_2\setminus P_+)\cup_\psi(P^3\times\{-1\})=(\tilde A_2)_{p,q}$. Thus, the manifold $A^T\cong L_{p,q}\#\tilde A_1\sqcup(\tilde A_2)_{p,q}$ and $L_{p,q}\setminus{\rm int}\,P^2_T\cong P^3$.
	In case 2), the proof is similar.
\end{proof}

A manifold $A$ is called {\it $E$-irreducible} if it does not have a summand $E$ in its connected sum decomposition and $\beta_1(A)\leqslant 1$. 

\begin{lemma}\label{Ene} If $A$ is an $E$-irreducible manifold, then any two-dimensional torus $T$ is compressible in it, and the manifold $A^T$ consists of two connected components, at least one of which is $E$-irreducible.
\end{lemma}
\begin{proof} Since $\beta_1(A)\leqslant 1$, the torus $T$ is compressible in $A$. Then, by Lemma \ref{decoo},  
	$A\cong\tilde A_1\#\tilde A_2$ and $A^T$ is homeomorphic to one of the listed manifolds:
\begin{itemize}
\item $A^T\cong L_{p,q}\#\tilde A_1\sqcup(\tilde A_2)_{p,q},p\neq 0,$
\item $A^T\cong E\#\tilde A_1\sqcup \tilde A_2, p=0.$
\end{itemize} 
Thus, the manifold $A^T$ contains an $E$-irreducible connected component, either $L_{p,q}\#\tilde A_1$, or $\tilde A_2$.
\end{proof}

\section{Topology of the filtration manifolds}
For the manifold $M_j$ of the filtration \eqref{fil}, let $Q_j=\partial M_j$. 

\begin{lemma}\label{compr} Each connected component of the manifold $Q_j,\,j\in\{1,\dots,k-1\}$ in the filtration \eqref{fil} is homeomorphic to $E$.
\end{lemma}
\begin{proof} From the decomposition \eqref{fil} it follows that $Q_j$ consists of $j$ copies of the manifold $E$ for $j\in\{1,\dots,k_0\}$ and $Q_{k-1}\cong E
	$. From formula \eqref{rouhan} it follows that for $j\in\{k_0+1,\dots,k-1\}$, $Q_{j-1}\cong Q_{j}^{T_j}$ for some torus  $T_j\subset Q_j$.

We will break down the further proof into steps.

{\it Step I. Let us show that the manifold $Q_j,j\in\{k_0+1,\dots,k-2\}$ does not contain $E$-irreducible connected components.}

Assuming the contrary, we obtain that for some $j\in\{k_0+1,\dots,k-2\}$, the manifold $Q_j$ contains an $E$-irreducible connected component $A$ and $T_j\subset A$. Since $Q_{j-1}\cong Q_{j}^{T_j}$, by Lemma \ref{decoo}, the manifold $Q_{j-1}$ contains an $E$-irreducible connected component. Continuing this process, we arrive at the conclusion that the manifold $Q_{k_0}$ has an $E$-irreducible connected component. This is a contradiction. 

{\it Step II. By induction on $i\in\{1,\dots,k_1\}$, we will show that each connected component of the manifold $Q_{k-i}$ is homeomorphic to $E$.} 

Since $Q_{k-1}\cong E$, the base of the induction is proven. Assume that each connected component of the manifold $Q_{k-i}$ is homeomorphic to $E$. Then the connected components of the manifold $Q_{k-i-1}$ are either homeomorphic to $E$ or, by Lemma \ref{decoo}, have one of the following forms:
\begin{itemize}
\item $L_{p,q}\sqcup (E)_{p,q}$;
\item  $L_{p,q}\# E\sqcup (\mathbb S^3)_{p,q}$;
\item $L_{p,q}\# (\mathbb S^3)_{p,q}$. 
\end{itemize}

According to \cite[Theorem 3.1]{Gui}, $H_1((\mathbb S^3)_{p,q})\cong\mathbb Z/q\mathbb Z$ and, consequently, the manifold $(\mathbb S^3)_{p,q}$ is $E$-irreducible for $q\neq 0$. From \cite[Remark, Corollary 8.3]{Gab} it follows that the manifold $(\mathbb S^3)_{1,0}$ is either homeomorphic to $E$ or $E$-irreducible. Since the manifolds $L_{p,q},p\neq 0$ and $L_{p,q}\# (\mathbb S^3)_{p,q},pq\neq 0$ are also $E$-irreducible, then, by Step 1, each connected component of the manifold $Q_{j-1}$ is homeomorphic to $E$.
\end{proof}

\section{Topology of the carrying manifold}
In this section, we will prove Theorem \ref{s3s1s}. To do this, note that in the filtration \eqref{fil}, $M^4\setminus{\rm int}\,M_{k-1}\cong P^4$. Then, by \cite{Max}), it suffices to show that $M_{k-1}\cong P^4$. This fact will follow from the following lemma.
\begin{lemma}\label{UiL} Each connected component of the manifold $M_j,\,j\in\{1,\dots,k-1\}$ in the filtration \eqref{fil} is homeomorphic to $P^4$.
\end{lemma}
\begin{proof} From the decomposition \eqref{fil} it follows that $M_j$ consists of $j$ copies of the manifold $P^4$ for $j\in\{1,\dots,k_0\}$. From formula \eqref{rouhan} it follows that for $j\in\{k_0+1,\dots,k-1\}$, the manifold $M_j$ is obtained from the manifold $M_{j-1}$ in the following way. Let $\bar M_{j-1}=M_{j-1}\sqcup (P^3\times\mathbb B^1)$ and $\varphi_j:P^3\times\mathbb S^0\to Q_{j-1}$ be an embedding such that $\varphi_j(P^3\times\{-1\})=P^{T_j}_1,\varphi_j(P^3\times\{1\})=P^{T_j}_2$ are the solid tori associated with $T_j$. Then
\begin{equation}\label{P123}
M_j\cong\bar M_{j-1}/_{\varphi_j}.
\end{equation}

By induction on $i\in\{0,\dots,k-1\}$, we will show that each connected component of the manifold $M_{k_0+i}$ is homeomorphic to $P^4$.

Since $M_{k_0}$ consists of $k_0$ copies of the manifold $P^4$, the base of the induction is proven. Assume that each connected component of the manifold $M_{k_0+i}=M_{j-1}$ is homeomorphic to $P^4$. Then the connected components of the manifold $M_j$ are either homeomorphic to $P^4$, or the component $U$ is obtained from the connected components of $M_{j-1}$ by the formula \eqref{P123}. 
Let $U_l,l\in\{1,2\}$ be the connected component of the manifold $M_{j-1}$ containing the solid torus $P^{T_j}_{l}$ and $A_l=\partial U_l$. According to Lemma \ref{compr}, $A_l\cong E$. By Lemma \ref{decoo}, the complement of one of the solid tori (assume for definiteness $P^{T_j}_{2}$) is homeomorphic to a solid torus, that is, $(A_2\setminus{\rm int}\,P_2^{T_j})\cong P^3$. Next, consider two possible cases: 1) $U_1\neq U_2$; 2) $U_1=U_2$.

In case 1) $U\cong U_1\cup (P^3\times[-1,1])\cup U_2$, where by the induction hypothesis $U_1\cong U_2\cong P^4$ and $P^{3}\times\{-1\}=P^{T_j}_1,P^{3}\times\{1\}=P^{T_j}_2$  (see Fig. \ref{dva}).  
\begin{figure}[ht] 
\centerline{\includegraphics[width
=0.7\textwidth]{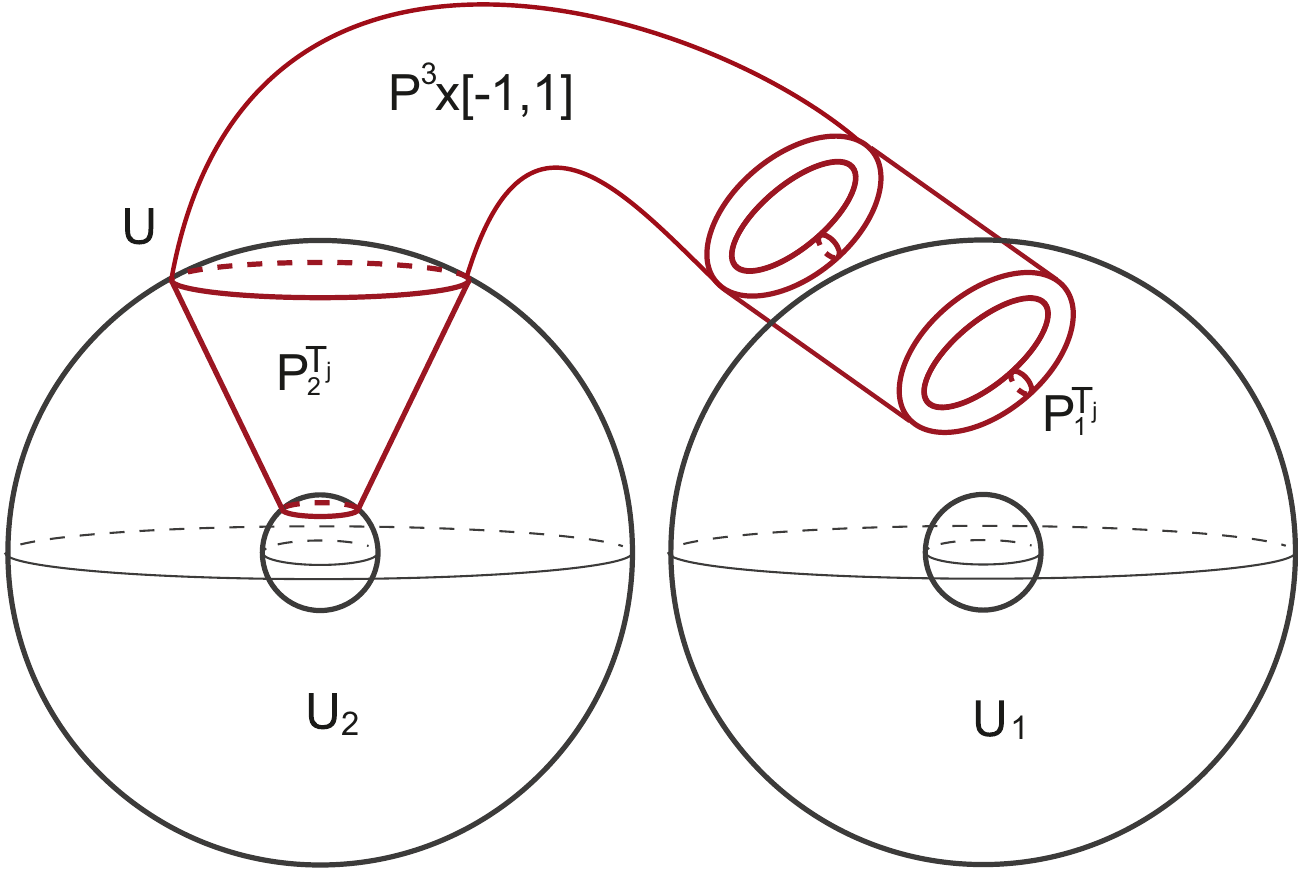}}
\caption{Illustration for the proof of Lemma \ref{UiL}, case 1)}
\label{dva}
\end{figure}
Since  $(A_2\setminus{\rm int}\,P_2^{T_j})\cong P^3$, then $U_2\cong P^3\times[1,2]$. Then $(P^3\times[-1,1])\cup U_2\cong P^3\times[-1,2]$ and, consequently, $U\cong U_1$.

In case 2) $U\cong (P^3\times[-1,1])\cup U_2$, where by the induction hypothesis $U_2\cong P^4$ and $P^{3}\times\{-1\}=P^{T_j}_1,P^{3}\times\{1\}=P^{T_j}_2$  (see Fig. \ref{odin}).  
\begin{figure}[ht] 
\centerline{\includegraphics[width
=0.6\textwidth]{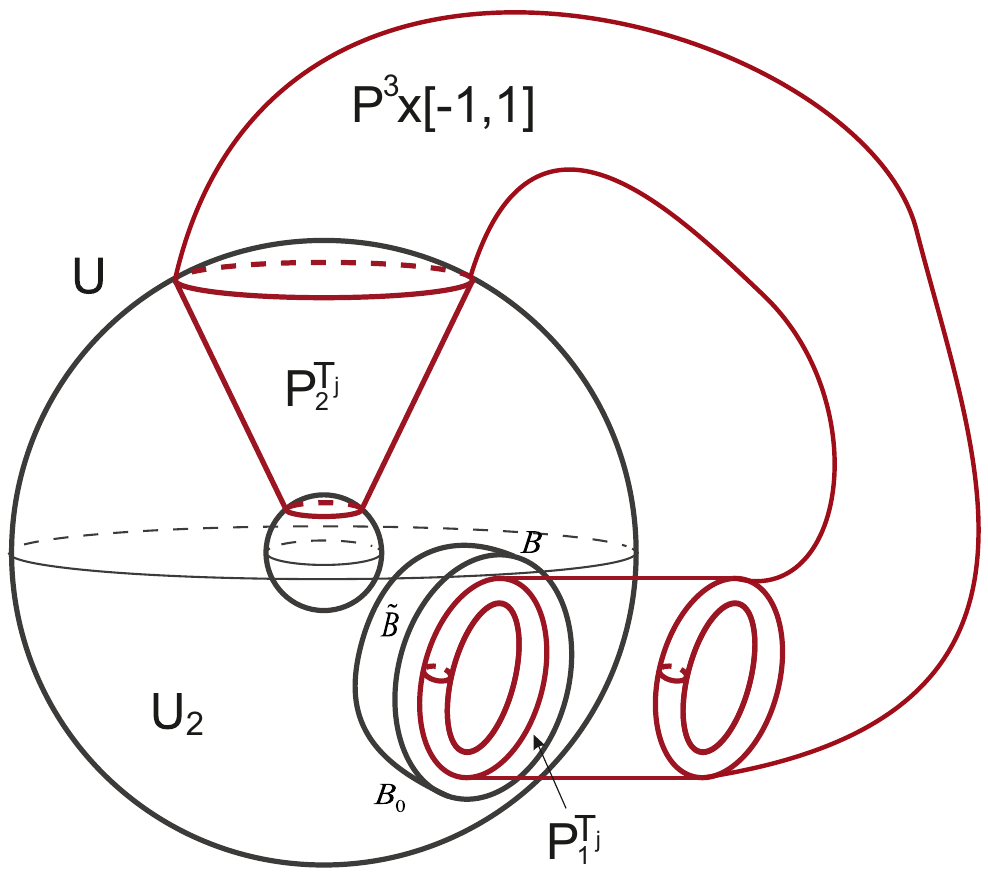}}
\caption{Illustration for the proof of Lemma \ref{UiL}, case 2)}
\label{odin}
\end{figure}
Since $(A_2\setminus{\rm int}\,P_2^{T_j})\cong P^3$, то $U_2\cong P^3\times[1,2]$. Since the solid tori $P_1^{T_j},P_2^{T_j}$ do not intersect the associated sphere $\Lambda_{T_j}$, it bounds a 3-ball $B$ in $U_2$ such that $P_1^{T_j}\subset {\rm int}\,B$ and $B\cap P_2^{T_j}=\varnothing$. 
Choose a 4-ball $\tilde B$ in $U_2$ such that $\tilde B\cap A_2=\partial\tilde B\cap A_2=B$. Then $B_0={\rm cl}\,(\partial\tilde B\setminus B)\cong\mathbb B^3$. Let $\tilde U_2={\rm cl}\,(U_2\setminus\tilde B)$ и $\tilde U={\rm  cl}\,(U\setminus B_0)$. By construction, $\tilde U$ is a connected 4-manifold with boundary, from which the manifold $U$ is obtained by identifying two disjoint copies of the 3-ball $B_0$ on its boundary. Then the proof of the theorem reduces to proving that $\tilde U\cong \mathbb B^4$.
 
To prove this fact, let us represent $\tilde U$ in the following form: $\tilde U=\tilde U_2\cup (P^3\times[-1,1]\cup\tilde B$. Since $\tilde U_2\cong U_2$, it follows that $\tilde U_2\cup(P^3\times[-1,1])\cong P^3\times[-1,2]$ and, consequently, $\tilde U\cong\tilde B$.
\end{proof}

\end{document}